\documentclass[runningheads]{llncs}

\usepackage{graphicx}
\usepackage{amssymb, amsmath}
\graphicspath{{figs/}}

\renewcommand{\paragraph}[1]{\smallskip\noindent\textbf{\textsf{#1}}}
\usepackage[inline]{enumitem}
\usepackage{csquotes}
\usepackage{subcaption}
\usepackage{todonotes}
\usepackage{booktabs}
\usepackage[inline]{enumitem}
\usepackage{subcaption}
\setlist[enumerate]{nosep}
\usepackage[colorlinks]{hyperref}
\usepackage{wrapfig}
\usepackage{xspace}

\newcounter{casecounter}
\newcounter{subcasecounter}

\DeclareMathOperator{\odd}{odd}
\DeclareMathOperator{\even}{even}

\hypersetup{colorlinks,breaklinks,
            urlcolor=[rgb]{.8,0.4,.8},
            linkcolor=[rgb]{0.3,0.3,.8},
            citecolor=[rgb]{.8,.2,.2}}

\newlength{\myx} 
\newlength{\myy} 
\newcommand\includegraphicstotab[2][\relax]{%
\settowidth{\myx}{\includegraphics[{#1}]{#2}}%
\settoheight{\myy}{\includegraphics[{#1}]{#2}}%
\parbox[c][1.1\myy][c]{\myx}{%
\includegraphics[{#1}]{#2}}%
}

\makeatletter
\newcommand{\ccase}[1]{%
  \stepcounter{casecounter}%
  \setcounter{subcasecounter}{0}%
  \protected@write \@auxout {}{\string \newlabel {#1}{{\thecasecounter}{\thepage}{\thecasecounter}{#1}{}} }%
  \hypertarget{#1}{\noindent\textbf{Case \thecasecounter.}}
}

\newcommand{\subcase}[1]{%
  \stepcounter{subcasecounter}%
  \protected@write \@auxout {}{\string \newlabel {#1}{{\thecasecounter.\thesubcasecounter}{\thepage}{\thecasecounter.\thesubcasecounter}{#1}{}} }%
  \hypertarget{#1}{\noindent\textbf{Case \thecasecounter.\thesubcasecounter.}}
}
\makeatother

\newcommand{\pc}{\ensuremath{\mathcal{P}}\xspace}

\newtheorem{observation}{Observation}

\newenvironment{pf}{\textit{Proof.~}}{}

\let\doendproof\endproof
\renewcommand\endproof{~\hfill$\qed$\doendproof}

\title{On Gallai's conjecture for series-parallel graphs and planar 3-trees\thanks{
  The work of P.~Kindermann and A.~Schulz was supported by DFG grant SCHU~2458/4-1.}}

\author{
Philipp Kindermann\inst{1}
\and Lena Schlipf\inst{1}
\and Andr\'e Schulz\inst{1}}
\institute{LG Theoretische Informatik,
  FernUniversit\"at in Hagen, Germany, 
  \email{ \{philipp.kindermann | lena.schlipf | andre.schulz\}@fernuni-hagen.de}
 }

\authorrunning{Kindermann, Schlipf, and Schulz}
\titlerunning{On Gallai's conjecture} 

\begin{document}
\maketitle

\begin{abstract}
A path cover is a decomposition of the edges of a graph into edge-disjoint simple paths.
Gallai conjectured that every connected $n$-vertex graph has a path cover with at most $\lceil n/2 \rceil$ paths.
We prove Gallai's conjecture for series-parallel graphs. For the class of planar 3-trees we show how to 
construct a path cover with at most $\lfloor 5n/8 \rfloor$ paths, which is an improvement over
the best previously known bound of $\lfloor 2n/3 \rfloor$.
\end{abstract}

\section{Introduction}
Let $G=(V,E)$ be a simple connected graph, with $|V|=n$ vertices. We say that~\pc is a \emph{path cover} of~$G$
if \pc is a collection of edge-disjoint simple paths, such that every edge of $G$ is contained in exactly 
one of the paths. The size of a path cover is the number of paths it contains.  
Following a question of Erd\H{o}s, Gallai conjectured that 
every simple connected graph with $n$ vertices has a path cover of size $\lceil n/2 \rceil$~\cite{L68}.
This conjecture is still open. It is easy to see that we cannot hope for a smaller bound, 
since a triangle needs a path cover of size~2.

In~1968, Lov\'asz proved that every simple connected graph can be covered with 
edge-disjoint simple paths and simple cycles such that the total number of paths and 
cycles does not exceed $\lfloor n/2 \rfloor$~\cite{L68}. In every odd-degree vertex, at least one path
has to have an endpoint. Hence, Lov\'asz result shows that Gallai's conjecture holds for the
case where every vertex of the graph has odd degree. Since we can always split a cycle into two paths,
it also follows that every graph has a path cover of size $n$. In subsequent work, the ideas
of Lov\'asz have been further exploited. Donald fixes an error in Lov\'asz' proof and also improves
the bound of the path cover size to $\lfloor 3n/4\rfloor$~\cite{D80}. Pyber showed that
Gallai's conjecture \emph{holds asymptotically}, which means that there is always a path
cover of size $n/2 + O(n^{3/4})$~\cite{P96}. In the same paper, it is also shown that if
a graph is $k$-connected then it has a path cover of size~$\lfloor n/2 \rfloor + \lceil n/2k \rceil$.
In 2000, Dean and Kouider proved that every graph has a path cover of size~$\lfloor 2n/3 \rfloor$~\cite{DK00},
which is currently the best known result for general graphs. 

If the graph is known to have a certain number of odd-degree vertices, then the above bounds can be improved.
In particular, if $G$ has~$n_\text{odd}$ odd-degree vertices and $n_\text{even}$ even-degree vertices, 
then the bound of Dean and Kouider can be  restated as $n_{\odd}/2+\lfloor 2n_{\even}/3\rfloor$.
Furthermore, it is known that Gallai's conjecture holds for graphs in which the graph induced by the 
even-degree vertices is a forest~\cite{P96}.

Gallai's conjecture was proven for certain graph classes. It is obvious that the conjecture holds for trees.
It has also been proven for outerplanar graphs~\cite{GFL14}. Recently, Bonamy and Perrett proved the conjecture 
for graphs with maximum degree~5~\cite{BP16}.

\begin{figure}[tb]
  \centering
  \includegraphics{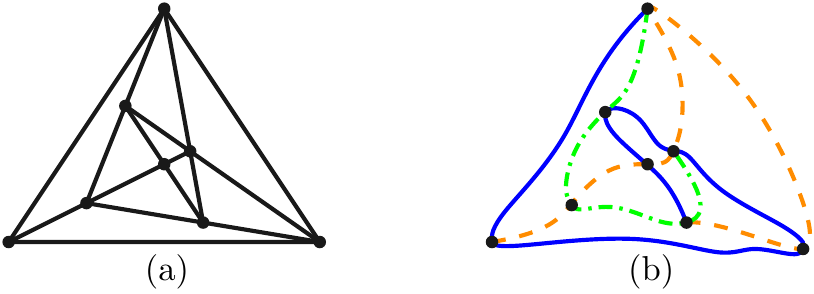}
  \caption{(a) A graph drawing with visual complexity (segments) of~9, and (b) The same graph with a path cover of size~3.}
  \label{fig:introexample}
\end{figure}

\paragraph{Motivation.} Our interest for Gallai's conjecture originates 
from an application in graph drawing. A recently proposed 
drawing criteria asks for the minimal number of geometric objects (e.g. subdivided straight-line segments, or
circular arcs) that are necessary to draw a planar graph as an arrangement~\cite{S15}. 
This quantity is known as the \emph{visual complexity} of a drawing.
Fig.~\ref{fig:introexample}
shows an example of a drawing with low visual complexity.
For most planar graph classes, known bounds for the visual complexity are not tight. The question that
arises from Gallai's conjecture can be seen as a graph-theoretic version of this problem.
In particular, any \enquote{arrangement drawing} induces a pseudoline arrangement and hence 
a path cover. However, not every path cover obtained in this way can be realized with paths drawn as 
straight polygonal chains~\cite{S90}. 
Any lower bound on the size of the path cover is obviously a lower bound for 
the visual complexity. Table~\ref{tbl:results} gives
an overview of the current bounds for both problems. Interestingly, for
some graph classes both problems have the same bound, but for other classes there is a significant difference between the graph-theoretic and the geometric version.

We could not find the lower bound for the series-parallel
graphs in the literature, but it is not difficult to see that the graph depicted in Fig.~\ref{fig:lb}
needs at least $\lfloor 3n/2\rfloor$ segments. In particular, if we align two edges, then it is impossible to \enquote{save}
further edges for the two incident triangles. Thus, for every two vertices we add, we might need three more 
segments.

\begin{figure}[b]
	\centering
	\includegraphics[scale=.7]{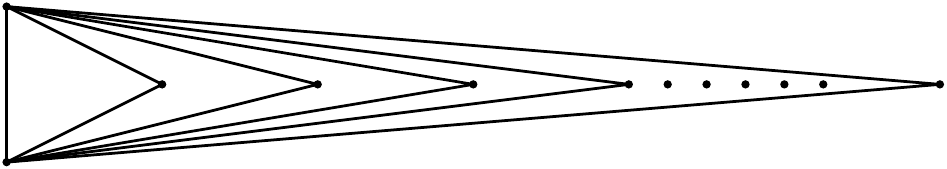}
	\caption{A series-parallel graph that needs $\lfloor 3n/2\rfloor$ segments when drawn as an arrangement.}
	\label{fig:lb}
\end{figure}

\paragraph{Results.} In the following Section~\ref{sec:sp}, we prove Gallai's conjecture for 
series-parallel graphs. Note that a graph is a partial 2-tree if and only if 
each biconnected component is a
series-parallel graph. We also make progress for the class of planar 3-trees by improving the
current bound for the necessary path cover size from $\lfloor 2n/3\rfloor$ 
to $\lfloor 5n/8\rfloor$. This result is presented in Section~\ref{sec:3trees}.

\begin{table}[t]
\centering
\caption{Bounds on the visual complexity (line segments) and on path cover sizes. 
Here, $n$ is the number of vertices, $n_\text{odd}$ the number of odd-degree vertices. Constant additions or subtractions have been omitted.}
\label{tbl:results}
\medskip
\begin{tabular}{l r p{.2cm}r p{8mm} r p{.2cm} c}
\toprule
\textbf{Class} & \multicolumn{3}{c}{\textbf{min. vis. compl. segments}} && \multicolumn{3}{c}{\textbf{min. size of path cover}}\\
 &  u.b. & & l.b && & u.b. &  \\
\midrule
Trees & $n_\text{odd}/2$ \cite{DESW07} && $n_\text{odd}/2$  & &$n_\text{odd}/2$  &&   \\
maximal outerplanar \ \  & $n$~\cite{DESW07}&& $n$~\cite{DESW07} & & $n/2$  && \cite{GFL14} \\
series parallel & $3n/2$~\cite{DESW07}  && $3n/2$ & & $\boldsymbol{n/2}$ && \textbf{Thm.~\ref{thm:seriesparallel}} \\
planar 3-trees & $2n$ \cite{DESW07} && $2n$ \cite{DESW07}  && $\boldsymbol{5n/8}$ && \textbf{Thm.~\ref{thm:planar3tree}} \\
cubic 3-connected & $n/2$  \cite{IMS15}&& $n/2$  &&  $n/2$ && \cite{BP16}  \\
triangulations & $7n/3$  \cite{DM14}&&$2n$ \cite{DESW07} && $ 2n/3$ && \cite{DK00}  \\
\bottomrule
\end{tabular}
\end{table}

\section{Series-parallel Graphs}
\label{sec:sp}
An \emph{$st$-graph} is an directed acyclic graph that has a unique source vertex~$s$ and a unique sink vertex~$t$.
A \emph{series-parallel} graph is an $st$-graph that is defined as follows.
\begin{enumerate}
  \item The graph $G=(\{s,t\},\{(s,t)\})$ is a series-parallel graph with
  	source $s$ and sink $t$.
  \item If $G_1$ is a series-parallel graph with source~$s_1$ and sink~$t_1$ 
  	and~$G_2$ is a series-parallel graph with source~$s_2$ and sink~$t_2$,
  	then the graph obtained by identifying~$t_1$ with~$s_2$ is a series-parallel
  	graph with source~$s_1$ and sink~$t_2$; this is called a 
  	\emph{series composition}.
  \item If $G_1$ is a series-parallel graph with source~$s_1$ and sink~$t_1$ 
  	and~$G_2$ is a series-parallel graph with source~$s_2$ and sink~$t_2$,
  	then the graph obtained by identifying~$s_1$ with~$s_2$ and~$t_1$ with~$t_2$
  	is a series-parallel graph with source~$s_1$ and sink~$t_1$; this is called
  	a \emph{parallel composition}.
\end{enumerate}
Although series-parallel graphs are defined as directed graphs, the paths in the 
path cover do not need to respect this orientation. In other words, we 
ignore the orientation of the edges once the series-parallel graph is constructed. 
A series-parallel graph is naturally associated with an ordered binary tree, 
called the \emph{SPQ-tree}. An SPQ-tree has three types of nodes:
\begin{description}
  \item[Q-node:] representing a single edge;
  \item[S-node:] a series composition between its children
    by identifying the sink of the left child with the source of the right child; and
  \item[P-node:] a parallel composition between its children.
\end{description}

A node of an SPQ-tree is a leaf if and only if it is a Q-node, and the number
of Q-nodes is exactly the number of edges in the underlying graph. Each SPQ-tree
represents a unique series-parallel graph, but there might be several SPQ-trees
for a given series-parallel graph. We only consider SPQ-trees with the property
that no S-node has an S-node as its left child and every P-node has an S-node as 
its left child. 
Such an SPQ-tree can be constructed in linear time as follows.

\begin{lemma}
  We can construct an SPQ-tree with the property that no S-node has an S-node as 
  its left child and every P-node has an S-node as its left child for
  every series-parallel graph with~$n$ vertices in $O(n)$ time.
\end{lemma}

\begin{proof}
  Let~$G$ be a series-parallel graph with~$n$ vertices. We first use the algorithm
  by Valdes et al.~\cite{VTL82} to construct \emph{some} SPQ-tree~$T=(V,E)$ of~$G$
  in $O(n)$ time. (Fig.~\ref{fig:SPQ1}+\ref{fig:SPQ2} show an example of a series-parallel graph $G$ and a SPQ-tree of $G$.) 
  	\begin{figure}[t]
	   \centering
		 \subcaptionbox{ \label{fig:SPQ1}}  
      {\includegraphics[width=0.25\textwidth,page=1]{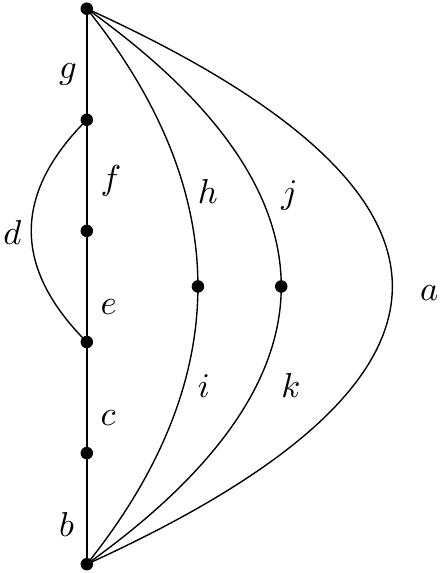}}
    \hfill
  \subcaptionbox{ \label{fig:SPQ2}}  
      {\includegraphics[scale=.75,page=1]{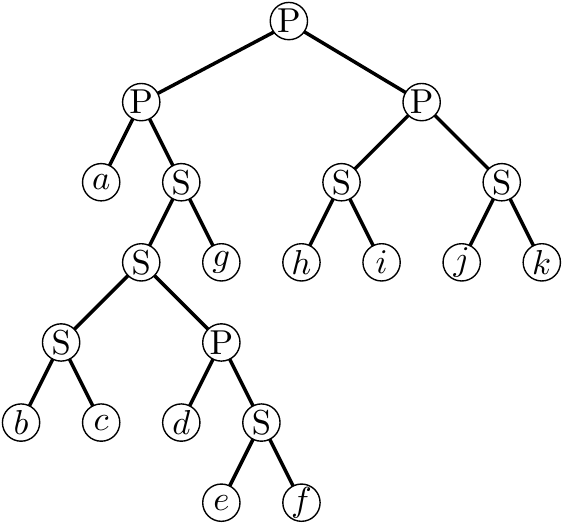} }
    \hfill
  \subcaptionbox{ \label{fig:SPQ3}}  
      {\includegraphics[scale=.75,page=1]{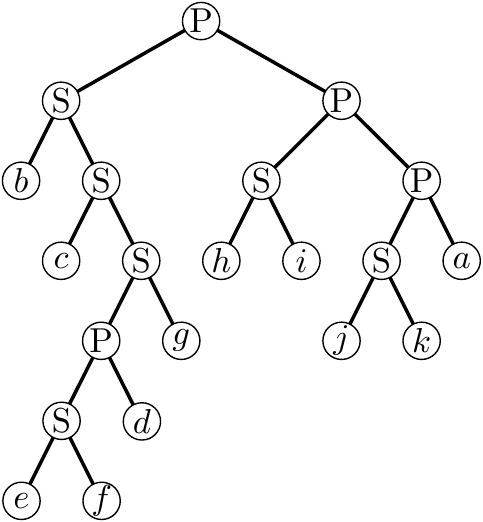} }
		\caption{(a) A series-parallel graph $G$, (b) an SPQ-tree $T$ of $G$,
      and (c) the SPQ-tree~$T^*$ for $G$ that has the property that
      no S-node has an S-node as its left child and every P-node has an S-node as 
      its left child.}
	\end{figure}
	
  We first create an SPQ-tree~$T'$ that represents~$G$ and has the property that
  no S-node has an S-node as its left child. Let~$S_1,\ldots,S_k$ be the 
  maximal connected components of~$T$ that contain only S-nodes and denote
  them by \emph{S-components}. Obviously, the S-components are pairwise 
  vertex-disjoint and there is no edge between any pair of S-components.
  We aim to create a sequence of 
  SPQ-trees $T_0,\ldots,T_k$ such that each SPQ-tree $T_i=(V,E_i), 0\le i\le k$
  \begin{enumerate}[label=(\roman*)]
    \item\label{prop:spq-ind1} has the same vertex set as~$T$, 
    \item\label{prop:spq-ind3} $S_1\ldots,S_k$ are the S-components of $T_i$, and
    \item\label{prop:spq-ind2} each component $S_1,\ldots,S_i$ in~$T_i$ is 
      a path that consists only of right edges, that is, edges from a parent to
      its right child.
  \end{enumerate}
  
  For~$T_0:=T$, the property holds trivially. Suppose that we have created an
  SPQ-tree~$T_i,0\le i\le k-1$ that satisfies this property. We create the
  SPQ-tree~$T_{i+1}$ as follows. Let $V'=\{v_1,\ldots,v_\ell\}$ be the vertex set 
  of~$S_{i+1}$ and let $U=\{u_1,\ldots,u_m\}$ be the children 
  of~$v_1,\ldots,v_\ell$ in $V\setminus V'$ such that both $v_1,\ldots,v_\ell$ 
  and $u_1,\ldots,u_m$ are ordered by an in-order traversal of~$T_i$. 
  By property~\ref{prop:spq-ind3}, all vertices in~$U$ are either P-nodes 
  or Q-nodes. The subtree of~$T_i$ induced by~$V'\cup U$ is a binary tree with 
  leaf set~$U$, so we have~$m=\ell+1$.
  Let~$v^*\in V'$ be the root of~$S_{i+1}$ in~$T_i$, let $T_i[v^*]$ be the maximal 
  subtree of~$T_i$ rooted in~$v^*$ and let~$T_i[u_j]$ be the maximal subtree 
  of~$T_i$ rooted in~$u_j,1\le j\le m$. 
  
  By property~\ref{prop:spq-ind3}, all vertices in $U$ are either P-nodes or
  Q-nodes. By the properties of an SPQ-tree, the subtree $T_i[v^*]$ represents
  a series composition of the subtrees $T_i[u_1],\ldots,T_i[u_m]$ in this order,
  that is, we can create the subgraph of~$G$ represented by $T_i[v^*]$ by doing
  a series composition on $T_i[u_1]$ with $T_i[u_2]$, then a series composition
  on the resulting graph with $T_i[u_3]$, and so on. We now create the 
  SPQ-tree~$T_{i+1}$ as follows: We remove all edges from the subtree induced
  by~$V'\cup U$. That leaves us with $2\ell+1$ connected components: 
  $\ell+1$ components for the induced subtrees $T_i[u_1],\ldots,T_i[u_{\ell+1}]$, 
  $\ell-1$ components each of which contains exactly one vertex of~$V'\setminus\{v^*\}$,
  and one component that contains the remaining vertices and has~$v^*$ as a 
  child. For $1\le j\le \ell$, we now connect~$u_j$ as the left child of~$v_j$
  and~$v_{j+1}$ as the right child of~$v_j$. Finally, we add~$u_{\ell+1}$ as
  the right child of~$v_{\ell}$; see Fig.~\ref{fig:S-component} for an example of this construction. By construction, the maximal induced subgraph
  $T_{i+1}[v^*]$ of $T_{i+1}$ has the same vertex set as~$T_i[v^*]$ and 
  represents the same graph; since we did not change the rest of~$T_i$, 
  property~\ref{prop:spq-ind1} is fulfilled for~$T_{i+1}$.
  All edges between vertices of~$S_{i+1}$ in~$T_{i+1}$
  are right edges, so property~\ref{prop:spq-ind2} is fulfilled for~$T_{i+1}$.
  We remove only edges incident to vertices of~$V'$, so we did not change
  the S-components $S_1\ldots,S_i,S_{i+2},\ldots,S_k$. Since all edges that we 
  added are either between two vertices of~$V'$ or
  between a vertex of~$V'$ and a P-node or a Q-node, we did not add an edge that
  connects~$S_{i+1}$ with any other S-component, so property~\ref{prop:spq-ind3}
  is fulfilled. We obtain the SPQ-tree~$T'$ of~$G$ by setting~$T'=T_k$.
  
	\begin{figure}[t]
  \centering
	\includegraphics[scale=.8]{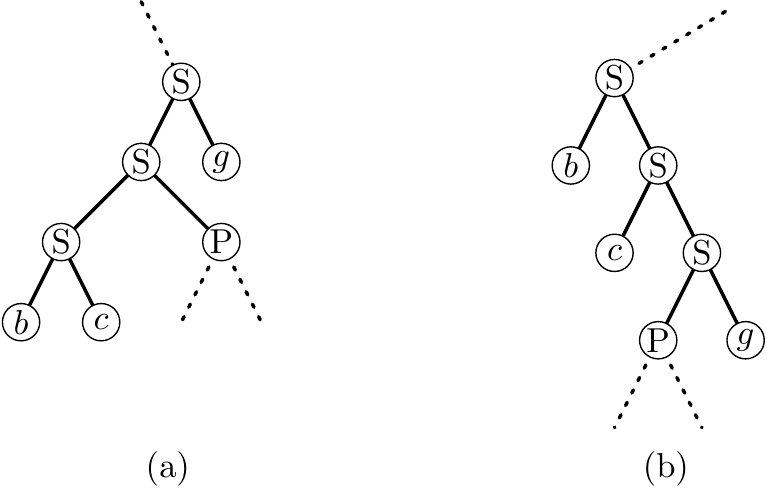}
	\caption{(a) An S-component in $T$ and (b) the modified component such that no $S$-node has an S-node as is left child.   }\label{fig:S-component}
	\end{figure}

  Since the leaves of an SQP-tree represent edges of the represented graph,
  there are~$2m-1$ vertices in~$T$ and, since the graph is series-parallel, $m\leq 2n-3$ (where $m$ is the number of edges). Hence, we can find the S-components in 
  $O(n)$ time by removing all edges that are not incident to two S-nodes and 
  taking the resulting connected components that contain $S$-nodes as S-components. We can fix an
  S-component of size $\ell$ in $O(\ell)$ time as described above because 
  our in-order traversals only have to traverse the vertices of the S-component
  and their children. Since the S-components are vertex-disjoint, we can thus
  fix all S-components in $O(n)$ time total, which yields the graph $T'$.
  
  We now create an SPQ-tree~$T^*$ that represents~$G$ and has the property that
  no S-node has an S-node as its left child and every P-node has an S-node as 
  its left child. The procedure works analogously as above by iteratively fixing
  the \emph{P-components} of~$T$, that is, the maximal connected components 
  $P_1,\ldots,P_\lambda$ of~$T$ that contain only P-nodes. The children of each
  P-component are all either S-nodes or Q-nodes; however, only one of them can
  be a Q-node, as there would be a multi-edge otherwise. In contrary to the
  S-nodes, changing the order of the children of a P-node does not change the
  graph. Hence, we can take any order on the children of a P-component when
  connecting them to the fixed P-component. In particular, we choose the 
  child that is a Q-node (if it exists) as the last child in this order. By
  this choice, it will be connected to the last P-node of the P-component
  (which is a path after fixing it) as a right child. Thus, all left children
  of the P-nodes are S-nodes; Fig.~\ref{fig:P-component} shows an example of this construction. The running time is the same as for fixing
  the S-components, which proves this lemma.
  
		\begin{figure}[t]
    \centering
	\includegraphics[scale=.8]{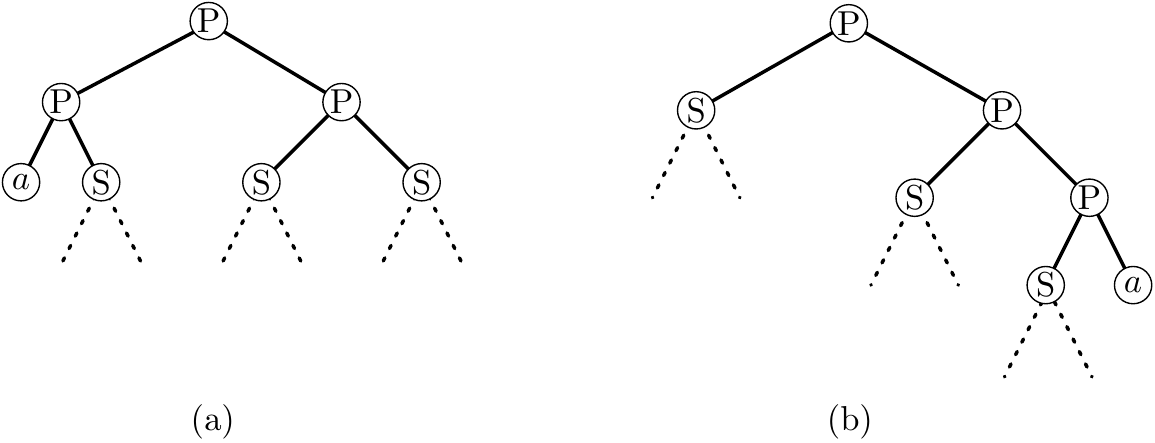}
	\caption{(a) A P-component in $T$ and (b) the modified component such that every P-node has an S-node as its left child.   }\label{fig:P-component}
	\end{figure}

\end{proof}

Given a series-parallel graph~$G$ with~$n$ vertices, we will build a path cover
of size at most~$\lceil n/2\rceil$ guided by
its SPQ-tree. 
Unfortunately, it is in general not
possible to combine a path cover of a series-parallel graph~$G_1$ with~$n_1$
vertices of size~$\lceil n_1/2\rceil$ and a path cover of a series-parallel 
graph~$G_2$ with~$n_2$ vertices of size~$\lceil n_2/2\rceil$ to a path cover
of its series or parallel composition~$G$ with~$n$ vertices of size~$\lceil n/2\rceil$. 
We create instead a path cover  
of size at most~$\lceil n/2\rceil + \rho$, where~$\rho$ is the number 
of specific substructures that
can later be used to reduce the number of paths to~$\lceil n/2\rceil$.

  We define a \emph{brace} of a path cover~$\pc$ as follows: 
  Let~$u$ and~$v$ be two vertices on a path~$P$ of~$\pc$
  and let~$P_3\subseteq P$ be the part of this path from~$u$ to~$v$.
  A brace~$B$ between~$u$ and~$v$ consists of~$P_3$
  and two more paths~$P_1$ and~$P_2$ from~$\pc$ that have~$u$ and~$v$ as their endpoints.
  The three paths are not allowed to share a vertex other than~$u$ and~$v$;
  see Fig.~\ref{fig:brace}a. We call the vertices in~$B$ different from~$u$
  and~$v$ \emph{interior vertices} and the paths $P_1$, $P_2$, and $P_3$ 
  \emph{interior disjoint}. In the following,~$\rho$
  denotes the number of braces a path cover has.

\begin{figure}[t]
\centering
\includegraphics{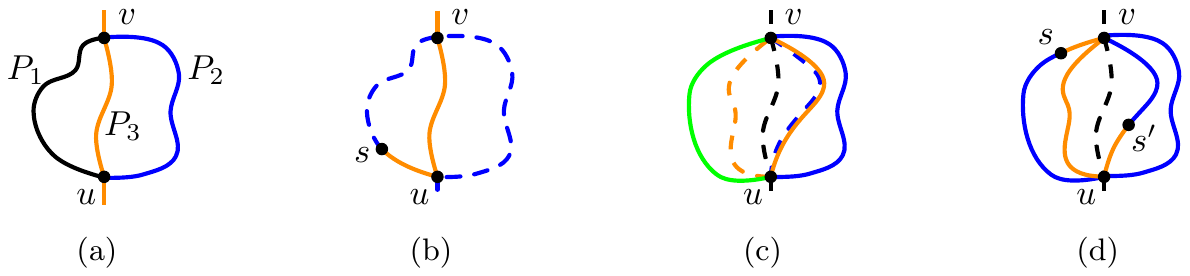}
\caption{(a) A brace between $u$ and $v$. (b) Removing a single brace. (c) $B$ (solid) and $B'$ (dashed) have a common path: the second one from the right. A common path of two braces is always the path $P_3$ from the brace that was created earlier. (d) Removing two braces. }
\label{fig:brace}
\end{figure}

We now define different types of path covers for series-parallel graphs.
Let~$G$ be a series-parallel graph with~$n$ vertices, source~$s$ and
sink~$t$. Table~\ref{tab:pathtypes} summarizes these types and their
number of paths.
\begin{itemize}
\item[$I_P$:] A path cover is of type~$I_P$ if it contains an $s$-$t$-path and 
if it has size at most~$n/2+\rho$.
\item[$I_S$:] A path cover is of type~$I_S$ if it contains an $s$-$t$-path and 
if it has size at most~$(n-1)/2+\rho$. Note that this type is the same 
as~$I_P$, but requires fewer paths.
\item[$O$:] A path cover is of type~$O$ if it contains two 
interior-vertex-disjoint $s$-$t$-paths and if it has size at 
most~$(n+1)/2+\rho$.
\item[$L$:] A path cover is of type~$L$ if it contains an $s$-$t$-path
and a path that starts in~$s$ and does not include~$t$ and if it has size at
most~$n/2+\rho$.
\item[$\Gamma$:] A path cover is of type~$\Gamma$ if it contains an $s$-$t$-path
and a path that starts in~$t$ and does not include~$s$ and if it has 
size at most~$n/2+\rho$. 
\end{itemize}

We group these types of path covers into two classes of 
types $\Pi=\{I_P,O\}$ and $\Sigma=\{I_S,L,\Gamma\}$.
We show next that each series-parallel graph admits a path cover of 
one of these types.

\begin{table}[tb]
  \centering
  \caption{The two classes of types of path covers, and their respective number 
    of paths minus the number of braces.}
  \label{tab:pathtypes}
  \begin{tabular}{@{\quad}l@{\quad}|@{\quad}c@{~~}cc@{\quad}|@{\quad}c@{\qquad}c@{\qquad}c@{\quad}}
    cover class & &$\Pi$& & & $\Sigma$ &  \\
    \hline
    path cover type & \includegraphicstotab[scale=.9]{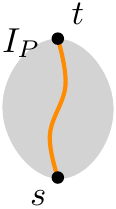} && \includegraphicstotab[scale=.9]{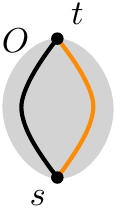}  &  \includegraphicstotab[scale=.9]{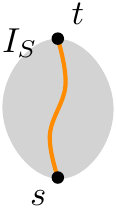}  & \includegraphicstotab[scale=.9]{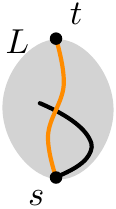} &\includegraphicstotab[scale=.9]{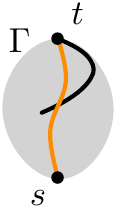}  \\
    \hline
    \hspace*{-2ex} max. \# of paths - $\rho\!\!\!\!$  & $n/2$ && $(n+1)/2$ & $(n-1)/2$ & $n/2$ & $n/2$
  \end{tabular}
\end{table}

\begin{lemma}\label{lem:sp-covertypes}
  Each series-parallel graph admits a path cover of type~$I_P$ if the root of
  its SPQ-tree is a Q-node, a path cover of a type in~$\Pi$ if the root of its
  SPQ-tree is a P-node, and a path cover of a type in~$\Sigma$ if the root of its
  SPQ-tree is an S-node.
	The path cover can be computed in linear time.
\end{lemma}
\begin{proof}
  Let~$G$ be a series-parallel graph with~$n$ vertices,~$m$ edges, source~$s$,
  and sink~$t$. We prove the lemma by induction on the number of Q-nodes~$m$
  in the SPQ-tree of~$G$.
  
  For~$m=1$, $G$ consists of exactly the two vertices~$s$ and~$t$ and an edge
  between them. The SPQ-tree of~$G$ consists of exactly one Q-node. We can cover 
  this edge with~1 path, which is a path cover of type~$I_P$. This is the only
  series-parallel graph with a Q-node as the root of its SPQ-tree.
  
  Now assume that we have shown the lemma for each graph with at most~$m-1$ 
  edges. In order to show that the lemma holds for~$G$, we distinguish two cases.

  \ccase{c:series} The root of the SPQ-tree of~$G$ is an S-node. 
  
  By construction
  of the SPQ-tree, the left child of the root is not an S-node. In the series
  composition, the order of the children matters since the sink of the left 
  child is identified with the source of the right child.
  Let~$G_1$ be the series-parallel graph with~$n_1$ vertices,~$m_1$ edges, $\rho_1$ braces, 
  source~$s_1$, and sink~$t_1$ represented by the SPQ-tree rooted in the left
  child of the root (which is not an S-node), and let~$G_2$ be the series-parallel graph 
  with~$n_2$ vertices,~$m_2$ edges, $\rho_2$ braces, source~$s_2$, and sink~$t_2$ represented by 
  the SPQ-tree rooted in the right child of the root. Since we have a 
  series composition, we have that $n=n_1+n_2-1$, $m=m_1+m_2$, $s=s_1$, $t_1=s_2$,
  and~$t=t_2$. Hence, $m_1,m_2<m$ and the lemma holds by induction for 
  both~$G_1$ and~$G_2$, that is, 
	$G_1$ has a path cover of a type in~$\Pi$ and $G_2$ has a path cover of a type in~$\Pi\cup\Sigma$.  
  We will now make a case analysis on the types of 
  path covers that~$G_1$ and~$G_2$ admit. The cases are summarized in 
  Table~\ref{tab:series}.
  
  \begin{table}[b]
    \centering
    \caption{The subcases of Case~\ref{c:series}: a series composition. A row represents
      a path cover type of~$G_1$, a column represents a path cover type of~$G_2$.
      A table entry contains the resulting path cover type of~$G$ and the case
      that handles this combination.}
    \label{tab:series}
    \begin{tabular}{@{\quad}l@{\quad}|@{\quad}ll@{\quad}ll@{\quad}ll@{\quad}ll@{\quad}ll@{\quad}}
               & $I_S$                  && $L$                    && $\Gamma$               && $I_P$                  && $O$\\
      \hline
      $I_P$    & $I_S$ & (\ref{sc:s-IX}) & $I_S$ & (\ref{sc:s-IX}) & $I_S$ & (\ref{sc:s-IX}) & $I_S$ & (\ref{sc:s-IX}) & $\Gamma$ & (\ref{sc:s-IO})\\
      $O$      & $L$   & (\ref{sc:s-OX}) & $L$   & (\ref{sc:s-OX}) & $L$   & (\ref{sc:s-OX}) & $L$   & (\ref{sc:s-OX}) & $I_S$ & (\ref{sc:s-OO})\\
    \end{tabular}
  \end{table}

 \begin{figure}[tb]
    \centering
    \begin{subfigure}[c]{0.15\textwidth}
      \centering
      \includegraphics[page=1]{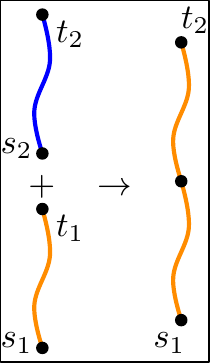}
      \subcaption{Case~\ref{sc:s-IX}: $I_P+I_P\rightarrow I_S$}
      \label{fig:CasesII}
    \end{subfigure}
    \hfill
    \begin{subfigure}[c]{0.15\textwidth}
      \centering
      \includegraphics[page=1]{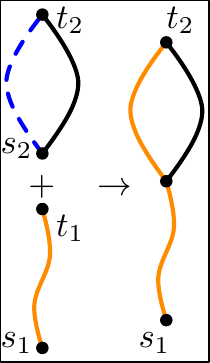}
      \subcaption{Case~\ref{sc:s-IO}: $I_P+O\rightarrow \Gamma$}
      \label{fig:CasesIO}
    \end{subfigure}  
    \hfill  
    \begin{subfigure}[c]{0.15\textwidth}
      \centering
      \includegraphics[page=1]{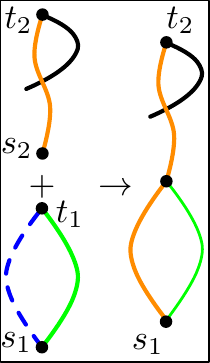}
      \subcaption{Case~\ref{sc:s-OX}: $O+\Gamma\rightarrow L$}
      \label{fig:CasesOI}
    \end{subfigure}
		 \hfill
        \begin{subfigure}[c]{0.15\textwidth}
      \centering
      \includegraphics[page=1]{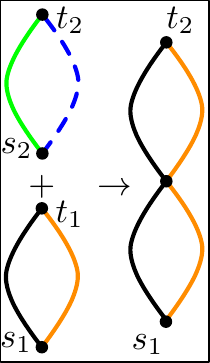}
      \subcaption{Case~\ref{sc:s-OO}: $O+O\rightarrow I_S$}
      \label{fig:CasesOO}
    \end{subfigure}
    \caption{Illustration for the subcases of Case~\ref{c:series}: a series composition.}
    \label{fig:Casesserial}
  \end{figure}

  \subcase{sc:s-IX} $G_1$ has a path cover of type~$I_P$ 
  of size a most~$n_1/2+\rho_1$, and~$G_2$ has a path cover of any type 
  from~$\Sigma\cup\Pi\setminus\{O\}$ of size at most~$n_2/2+\rho_2$.
  
  We obtain a path cover of 
  type~$I_S$ for~$G$ with $\rho=\rho_1+\rho_2$ braces of size at most
  $n_1/2+\rho_1+n_2/2+\rho_2-1=(n+1)/2+\rho-1=(n-1)/2+\rho$ by merging
  the $s_1$-$t_1$-path from~$G_1$ with the $s_2$-$t_2$-path from~$G_2$ at
  $t_1=s_2$ (note that every type contains at least one path from its source to
  its sink); see Fig.~\ref{fig:CasesII}. Since $s=s_1$ and $t=t_2$, we can choose the 
  merged path as the $s$-$t$-path from~$G$.
  
  \subcase{sc:s-IO} $G_1$ has a path cover of type~$I_P$ 
  of size at most~$n_1/2+\rho_1$, and~$G_2$ has a path cover of type~$O$ with at most~$(n_2+1)/2+\rho_2$ paths.
  
  We obtain a path cover of type~$\Gamma$ for~$G$ with $\rho=\rho_1+\rho_2$ 
  braces of size at most $n_1/2+\rho_1+(n_2+1)/2+\rho_2-1=(n+2)/2+\rho-1=n/2+\rho$ 
  by merging the $s_1$-$t_1$-path from~$G_1$ with one of the 
  $s_2$-$t_2$-paths from~$G_2$ at $t_1=s_2$; see Fig.~\ref{fig:CasesIO}. We choose
  this merged path as the $s$-$t$-path from~$G$ and the other $s_2$-$t_2$-path
  from~$G_2$ as the path from~$G$ that starts in $t=t_2$ and does not include~$s$.
  
  \subcase{sc:s-OX} $G_1$ has a path cover of type~$O$
  of size at most~$(n_1+1)/2+\rho_1$, and~$G_2$ has a path cover of any type 
  from~$\Sigma\cup\Pi\setminus\{O\}$ with at 
  most~$n_2/2+\rho_2$ paths.
  
  We obtain a path cover of type~$L$ for~$G$ with $\rho=\rho_1+\rho_2$ 
  braces of size at most $(n_1+1)/2+\rho_1+n_2/2+\rho_2-1=(n+2)/2+\rho-1=n/2+\rho$ 
  by merging one of the $s_1$-$t_1$-paths from~$G_1$ with the 
  $s_2$-$t_2$-path from~$G_2$ at $t_1=s_2$; see Fig.~\ref{fig:CasesOI}. We choose
  this merged path as the $s$-$t$-path from~$G$ and the other $s_1$-$t_1$-path
  from~$G_1$ as the path from~$G$ that starts in $s=s_1$ and does not include~$t$.
  
  \subcase{sc:s-OO} $G_1$ has a path cover of type~$O$ 
  of size at most~$(n_1+1)/2+\rho_1$, and~$G_2$ has a path cover of type~$O$ 
  of size at most~$(n_2+1)/2+\rho_2$.
  
  We obtain a path cover of type~$I_S$ for~$G$ with $\rho=\rho_1+\rho_2$ 
  braces of size at most $(n_1+1)/2+\rho_1+(n_2+1)/2+\rho_2-2=(n+3)/2+\rho-2=(n-1)/2+\rho$ 
  by merging one of the $s_1$-$t_1$-paths from~$G_1$ with one of the 
  $s_2$-$t_2$-paths from~$G_2$ at $t_1=s_2$ and the other $s_1$-$t_1$-path 
  from~$G_1$ with the other $s_2$-$t_2$-path from~$G_2$ at $t_1=s_2$; see 
  Fig.~\ref{fig:CasesOO}. We choose one of the merged paths as the $s$-$t$-path 
  from~$G$. Note that we cannot use both $s$-$t$-paths for an $O$-configuration 
  as they are not interior-vertex-disjoint (they both include $s_2=t_1$).

  This covers all combinations of series compositions.
  
  \ccase{c:parallel} The root of the SPQ-tree of~$G$ is a P-node. 
  
  By construction
  of the SPQ-tree, the left child of the root is an S-node.
  Let~$G_1$ be the series-parallel graph with~$n_1$ vertices,~$m_1$ edges, $\rho_1$ braces,
  source~$s_1$, and sink~$t_1$ represented by the SPQ-tree rooted in the 
  left child of the root (which is an S-node),
  and let~$G_2$ be the series-parallel graph 
  with~$n_2$ vertices,~$m_2$ edges, $\rho_2$ braces, source~$s_2$, and sink~$t_2$ represented by 
  the SPQ-tree rooted in the
  right child of the root. Since we have a 
  parallel composition, we have that $n=n_1+n_2-2$, $m=m_1+m_2$, $s=s_1=s_2$,
  and~$t=t_1=t_2$. Hence, $m_1,m_2<m$ and the lemma holds by induction for 
  both~$G_1$ and~$G_2$, that is, $G_1$ has a path cover of a type in~$\Sigma$
  and~$G_2$ has a path cover of a type in~$\Sigma\cup\Pi$. We will now make a 
  case analysis on the types of 
  path covers that~$G_1$ and~$G_2$ admit. The cases are summarized in 
  Table~\ref{tab:parallel}.
  
  \subcase{sc:p-II} $G_1$ has a path cover of type~$I_S$ 
  of size at most~$(n_1-1)/2+\rho_1$ and~$G_2$ has a path cover of type~$I_S$ or~$I_P$. 
  
  Since every path cover of type~$I_S$ is automatically 
  a path cover of  type~$I_P$, we assume that~$G_2$ has a path cover of 
  type~$I_P$ of size at most~$n_2/2+\rho_2$. We obtain a path cover of type~$O$ 
  for~$G$ with $\rho=\rho_1+\rho_2$ braces of size at most
  $(n_1-1)/2+\rho_1+n_2/2+\rho_2=(n+1)/2+\rho$ by combining the path 
  covers of~$G_1$ and~$G_2$; see Fig.~\ref{fig:CasepII}. We choose the $s$-$t$-path
  from~$G_1$ and the $s$-$t$-path from~$G_2$ as the two $s$-$t$-paths from~$G$.
  
  \subcase{sc:p-IO} $G_1$ has a path cover of type~$I_S$ 
  of size at most~$(n_1-1)/2+\rho_1$ and~$G_2$ has a path cover of type~$O$
  with at most $(n_2+1)/2+\rho_2$ paths. 
  
  We obtain a path cover of 
  type~$I_P$ for~$G$ with $\rho=\rho_1+\rho_2+1$ braces of size at most
  $(n_1-1)/2+\rho_1+(n_2+1)/2+\rho_2=(n+2)/2+\rho-1=n/2+\rho$ by 
  adding a brace between~$s$ and~$t$ that consists of the $s$-$t$-path 
  from the path cover of~$G_1$ and of the two $s$-$t$-paths from the path cover
  of~$G_2$; see Fig.~\ref{fig:CasepIO}. We choose the $s$-$t$-path from~$G_1$ as the 
  $s$-$t$-path from~$G$ and also as the path~$P_3$ from the new brace. Hence, the
  other two paths~$P_1$ and~$P_2$ of the new brace will not be changed in the
  future.
  
  \subcase{sc:p-IL} $G_1$ has a path cover of a type in~$\Sigma$ 
  of size at most~$n_1/2+\rho_1$ and~$G_2$ has a path cover of type~$L$ or $\Gamma$
   of size at most $n_2/2+\rho_2$.
  
  We consider the case that~$G_2$ has the $\Gamma$-configuration, the other 
  case works symmetrically. We obtain 
  a path cover of type~$I_P$ for~$G$  with $\rho=\rho_1+\rho_2$ braces of size 
  at most $n_1/2+\rho_1+n_2/2+\rho_2-1=(n+2)/2+\rho-1=n/2+\rho$ by 
  merging two paths: we take the $s$-$t$-path from~$G_1$ and the path from~$G_2$ 
  that starts in~$t$ and does not include~$s$. Because of this property, these 
  two paths share exactly the vertex~$t$. We merge these two paths at~$t$, 
  essentially reducing the number of paths by~1; see Fig.~\ref{fig:CasepLGamma}. We choose 
  the $s$-$t$-path from~$G_2$ as the $s$-$t$-path from~$G$. 
  
  \subcase{sc:p-LI} $G_1$ has a path cover of type~$L$ or~$\Gamma$ and~$G_2$
  has a path cover of type~$I_S$ or~$I_P$. This case works analogously to 
  Case~\ref{sc:p-IL}.
  
  \subcase{sc:p-LO} $G_1$ has a path cover of type~$L$ or~$\Gamma$ 
   of size at most $n_1/2+\rho_1$ and~$G_2$ has a
  path cover of type~$O$ of size at most $(n_2+1)/2+\rho_2$.
  
  We will consider the case that~$G_1$ has the $L$-configuration, the other 
  case works symmetrically. We obtain a path cover of type~$O$ for~$G$ with 
  $\rho=\rho_1+\rho_2$ braces of size 
  at most $n_1/2+\rho_1+(n_2+1)/2+\rho_2-1=(n+3)/2+\rho-1=(n+1)/2+\rho$ by 
  merging two paths: we take one $s$-$t$-path 
  from~$G_2$ and the path from~$G_1$ that starts in~$s$ and does not include~$t$. 
  As in previous cases, we merge these two paths at~$s$; 
  see Fig.~\ref{fig:CasepLO}. We choose the remaining $s$-$t$-path of~$G_2$ and the
  $s$-$t$-path of~$G_1$ as the two $s$-$t$-paths from~$G$.

  This covers all combinations of parallel compositions.
	
	For the run time, we first observe that storing at each node of the tree the type of path cover the corresponding subtree has only takes additionally constant time. We traverse the SPQ-tree bottom-up. Deciding which case to apply and (if needed) merging two paths  takes constant time. Thus, the algorithm runs in linear time.
\end{proof}

  \begin{table}[tb]
    \centering
    \caption{The subcases of Case~\ref{c:parallel}: a parallel composition. A row represents
      a path cover type of~$G_1$, a column represents a path cover type of~$G_2$.
      A table entry contains the resulting path cover type of~$G$ and the case
      that handles this combination. The table entry marked by a star means
      that we create an additional brace while handing the parallel composition.}
    \label{tab:parallel}
    \begin{tabular}{@{\quad}l@{\quad}|@{\quad}ll@{\quad}ll@{\quad}ll@{\quad}ll@{\quad}ll@{\quad}}
               & $I_S$                  && $L$                    && $\Gamma$               && $I_P$                  && $O$ \\
      \hline
      $I_S$    & $O$   & (\ref{sc:p-II}) & $I_P$ & (\ref{sc:p-IL}) & $I_P$ & (\ref{sc:p-IL}) & $O$   & (\ref{sc:p-II}) & $I_P$ * & (\ref{sc:p-IO}) \\
      $L$      & $I_P$ & (\ref{sc:p-LI}) & $I_P$   & (\ref{sc:p-IL})  & $I_P$ & (\ref{sc:p-IL}) & $I_P$ & (\ref{sc:p-LI}) & $O$     & (\ref{sc:p-LO}) \\
      $\Gamma$ & $I_P$ & (\ref{sc:p-LI}) & $I_P$ & (\ref{sc:p-IL}) & $I_P$   & (\ref{sc:p-IL}) & $I_P$ & (\ref{sc:p-LI}) & $O$     & (\ref{sc:p-LO}) \\
    \end{tabular}
  \end{table}

We can now use this lemma to show that any series-parallel graph~$G$ with~$n$
vertices admits a path cover of size at most~$\lceil n/2\rceil$.

\begin{theorem}\label{thm:seriesparallel}
  Any series-parallel graph~$G$ with~$n$
  vertices admits a path cover of size at most~$\lceil n/2\rceil$. The path cover can be computed in linear time.
\end{theorem}
\begin{proof}
  We use Lemma~\ref{lem:sp-covertypes} to obtain a path cover of~$G$ with~$\rho$
  braces of size at most~$(n+1)/2+\rho$. It remains to show that we can use
  the braces to reduce the total number of paths by~$\rho$.
  
  Note that the only operations we use in the proof of Lemma~\ref{lem:sp-covertypes}
  is merging existing paths or creating new paths, but we never split a path.
  This means that the internal structure of the braces remains untouched,
  that is, the paths~$P_1$,~$P_2$, and $P_3$ are not split in
  the resulting path cover. Further, after creating a brace, we never use the 
  paths~$P_1$ and~$P_2$ as an $s$-$t$-path for the series and parallel 
  composition, as we use the path~$P_3$ as the only designated $s$-$t$-path
  in the resulting path cover; see Case~\ref{sc:p-IO}. Hence, every 
  path~$P_1$ and~$P_2$ is also a path in the resulting path cover.
  Let~$B_1$ be a brace between vertices~$u_1$ and~$v_1$
  and let~$B_2$ be a brace between vertices~$u_2$ and~$v_2$.
  By the structure
  of series-parallel graphs, one of the following holds.
  \begin{enumerate}[label=(\roman*)]
    \item $B_1$ and $B_2$ are \emph{independent}, that is, $u_1$ and~$v_1$ do not
      lie in~$B_2$, and $u_2$ and~$v_2$ do not lie in~$B_1$;
    \item $B_1$ and $B_2$ are \emph{parallel}, that is, $u_1=u_2$ and $v_1=v_2$;
    \item $B_1$ \emph{is included in} $B_2$ (or vice versa), that is,~$u_1$ and~$v_1$ lie
      in~$B_2$, but~$u_2$ and/or~$v_2$ do not lie in~$B_1$; or
    \item $B_1$ and $B_2$ are \emph{consecutive}, that is, $u_2=v_1$ or $u_1=v_2$.
  \end{enumerate}
  
  If $B_1$ is included in $B_2$, then we write $B_1 \lessdot B_2$.
  The parallel-relation of the braces forms an equivalence relation.
  We denote the corresponding equivalence classes by~$S_1,\ldots,S_k$.
  The sets are ordered according to the (linear extension of the) 
  partial order as induced by the included-relation.
  In particular, for all $1\le i<j\le k$, we cannot find any $B\in S_i$ and  $B'\in S_j$ 
  such that $B \lessdot B'$.
  
 \begin{figure}[tb]
      \centering    
  \begin{subfigure}[c]{0.3\textwidth}
      \centering
      \includegraphics[page=3]{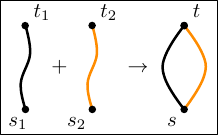}
      \subcaption{Case~\ref{sc:p-II}: $I_S + I_P \rightarrow O$}
      \label{fig:CasepII}
    \end{subfigure}
    \hfil
    \begin{subfigure}[c]{0.3\textwidth}
      \centering
      \includegraphics[page=3]{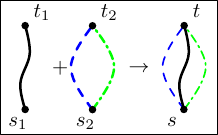}
      \subcaption{Case~\ref{sc:p-IO}: $I_S + O \rightarrow I_P$}
      \label{fig:CasepIO}
    \end{subfigure}
    \\[5pt]
    \begin{subfigure}[c]{0.3\textwidth}
      \centering
      \includegraphics[page=4]{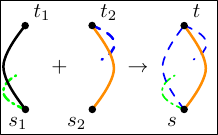}
      \subcaption{Case~\ref{sc:p-IL}: $L + \Gamma \rightarrow I_P$}
      \label{fig:CasepLGamma}
    \end{subfigure}
     \hfil
    \begin{subfigure}[c]{0.3\textwidth}
      \centering
      \includegraphics[page=4]{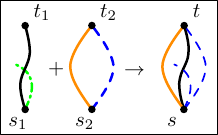}
      \subcaption{Case~\ref{sc:p-LO}: $L + O \rightarrow O$}
      \label{fig:CasepLO}
    \end{subfigure}
    \caption{Illustration for the subcases of Case~\ref{c:parallel}: a parallel composition.}
    \label{fig:Casesparallel}
  \end{figure}

  Let~$B$ be a brace between vertices~$u$ and~$v$ with paths~$P_1$, $P_2$, 
  and~$P_3$. We denote with~$\mathcal{B}$ the set of braces that are included in~$B$
  and contain at least one interior vertex of~$P_1$ or~$P_2$. A
  \emph{split vertex} of~$B$ is a vertex in the interior of~$P_1$ or~$P_2$ 
  that is not an interior vertex in any brace of~$\mathcal{B}$. We show
  that $B$ always has a split vertex. If $\mathcal{B}=\emptyset$,
  then there has to be a split vertex, since otherwise~$P_1$ and $P_2$ would
  form a multiple edge. Otherwise, take a brace~$B'\lessdot B$ for
  which there is no brace~$B''$ with $B' \lessdot B'' \lessdot B$. Assume that~$B'$
  lies between~$u'$ and~$v'$. Since $B'\lessdot B$, we know that $B'$ and $B$
  are not parallel. In particular, $u=u'$ and $v=v'$ is impossible. W.l.o.g., we
  can assume that $u\neq u'$. In this case,~$u'$ is a split vertex since it is by definition contained
  in $P_1$ or $P_2$ and by construction it is not an interior vertex of  
  $B'$; therefore, it is also not an interior vertex in any brace of~$\mathcal{B}$.
  We call the path
  of $P_1$ and $P_2$ that contains the split vertex the \emph{split path}
  and the other one the \emph{unsplit path}.
  
  We remove the sets of braces in order $S_1,\ldots,S_k$, starting with $S_1$. Assume that
  we already removed the  braces in 	$S_1,\ldots,S_{i-1}$. 
  If~$|S_i|>1$, we first sort the braces in $S_i$ in their creation order starting with the brace that was created first. 
  We remove the braces according to this order. 
  Let~$B$ and~$B'$ be the next two braces to remove from~$S_i$ between vertices~$u$ and~$v$.
  Furthermore, let $s$ be a split vertex in $B$ and let $s'$ be a split vertex in $B'$. 
  The paths of~$B$ and the paths of~$B'$ are either pairwise interior-vertex-disjoint,
  or---since~$B$ was created before~$B'$---the path~$P_3$ of~$B$ is equivalent
  to one of the paths of~$B'$. However, the paths~$P_1$ and~$P_2$ of~$B$ are 
  interior-vertex-disjoint from the paths of~$B'$ as we kept $P_3$ as the
  only designated $s$-$t$-path in the resulting path cover, so we have $s\neq s'$.
  We now delete the split
  paths and the unsplit paths of~$B$ and~$B'$ from the path cover. Then, we create a new path
  that starts in~$s$, follows the split path of~$B$ 
  until~$v$, follows then the unsplit path of~$B'$ until~$u$, and continues with the
  split path of~$B'$ until~$s'$. We create a second new path that starts in~$s$,
  follows the split path of~$B$ until~$u$, follows then the unsplit path of~$B$
  until~$v$, and continues with the split path of~$B'$ until~$s'$; see Fig.~\ref{fig:brace}c--d.
  We then remove~$B$ and~$B'$ from~$S_i$.
  Note that the paths we are changing are of \enquote{type} $P_1$ or $P_2$ 
  in either $B$ or $B'$. Hence, they all have $u$ and~$v$
  as endpoints and we do not create new paths by removing them.
  In the end, we removed two braces from the
  path cover and reduced its number of paths by~2. We repeat this step 
  until~$|S_i|\le 1$. 
  
  If~$|S_i|=1$, let~$B$ be the only brace in~$S_i$ that lies between~$u$ and~$v$ and that has  
  split vertex~$s$. 
  Note that even if we have removed other braces from $S_i$ beforehand and $B$ is the last brace in this set,  the paths $P_1$, $P_2$,    and $P_3$ of $B$ remained untouched so far. This is ensured by resolving the braces in creation order.
  We now remove
  both the split path and the unsplit path of~$B$ from the path cover. Then,
  we split the path~$P_3$ from~$B$ at~$u$ into two paths~$P$ and~$P'$ such 
  that~$v$ lies on~$P'$. We then extend~$P'$ along the split path of~$B$ until~$s$
  and~$P$ along the unsplit path of~$B$ until~$v$, and then along the split path
  until~$s$; see Fig.~\ref{fig:brace}b.  
  By this, we removed one brace and reduced the number of paths by~1.

  With this method, we can remove all braces in~$S_i$ from the 
  path cover and reduce its number of paths by~$|S_i|$.
  At the end of this procedure, we have removed all~$\rho=\sum_{i=1}^k |S_i|$ braces and~$\rho$
  paths from the path cover. Thus, we obtain a path cover of size at 
  most~$(n+1)/2\le \lceil n/2\rceil$.

	It remains to prove the runtime. First we use Lemma~\ref{lem:sp-covertypes} to obtain in linear time a path cover of~$G$ with~$\rho$ braces. 
 On the fly we mark all nodes of the tree that create braces. 
 Parallel braces lie in the same P-component, thus it is easy to find the equivalent classes of the braces. 
 If $B_1$ is included in $B_2$ then the node creating $B_1$ lies in the subtree of the node 
 creating $B_2$ (the converse statement does not hold).  
Thus, in order to remove the sets of braces in the correct order, we simply have to traverse the tree top-down. 
To ensure that parallel braces are removed in creation order, we have to remove the braces in the same P-component bottom-up. 
Since every P-component is a path the removal order of the braces can be computed in linear time.

When removing the braces, we have to find the corresponding split vertex. Thus, we have to find a vertex for every 
brace $B$ that is not an interior vertex of a brace $B'$ with $B'\lessdot B$ efficiently. 
This information can be precomputed as follows.
We store in every interior node~$x$ a pointer to a \emph{potential split vertex}~$z$,
that is, an interior vertex on its $s$-$t$-path with the following property:
If $x$ is an S-node, then $z$ lies currently in no brace. If~$x$ is a P-node,
then any brace that currently contains~$z$ is a brace between~$s$ and~$t$.
Further, we store for every brace a pointer to a split vertex.

We now go through the tree bottom-up and compute all potential split vertices and split vertices.
If a node is an S-node, we select the combined $t_1=s_2$~vertex of its children
as potential split vertex. By construction this vertex cannot lie in any brace at this point.
If it is a P-node we do the following. Let $z_1$ be the potential split vertex of the left child
and $z_2$ be the the potential split vertex of the right child. 
The $s$-$t$-path of the left child is always an $s$-$t$-path of the P-node. Hence 
we can pick~$z_1$ as the potential split vertex for the P-node. If the parallel operation 
introduces a brace we store~$z_2$ as split vertex of the brace. 
Braces will only be introduced in Case~\ref{sc:p-IO} and the two $s$-$t$-paths
from the $O$-configuration are chosen as the paths~$P_1$ and~$P_2$.
Hence~$z_2$ lies on~$P_1$ or $P_2$ and splits the corresponding path appropriately.
This concludes the proof.
\end{proof}

\section{Planar 3-Trees}
\label{sec:3trees}
A  graph is a \emph{planar 3-tree}, also known as an \emph{Apollonian network}, 
if it can be constructed from the $K_3$
and a sequence of \emph{stacking operations}.
A stacking operation adds a vertex $v$ to the graph by selecting an interior triangular face $abc$
and introducing the edges $va$, $vb$ and $vc$. The graph loses the face $abc$ but wins the
faces $vab$, $vac$, $vbc$; see Fig.~\ref{fig:planar3tree}. 
We assume from now on that~$G$ is a planar 3-tree with~$n$ vertices.
The graph~$G$ can be associated  
with an ordered rooted ternary tree~$T_G$ as follows. 
The interior vertices of~$G$ are in 1--1 correspondence to the interior nodes of~$T_G$,
and the interior faces of~$G$ are in 1--1 correspondence with the leaves of~$T_G$.
If $G$ is the $K_3$, then $T_G$ has a single node, which is a leaf.
When stacking a vertex into some face~$f$, we attach three children to
the leaf that was associated with~$f$ (and relabel the former leaf with $f$). We also choose an appropriate convention 
for the order of the leaves that allows us to identify faces and leaves. 
The tree $T_G$ is obtained by carrying
out all of $G$'s stacking operations this way. 
Note that this tree can be labeled
such that it gives a tree decomposition of~$G$ with bags of size~4. Therefore, planar 3-trees have treewidth~3.

\begin{figure}[t]
  \subcaptionbox{\label{fig:planar3tree}}{\includegraphics[scale=.8,page=1]{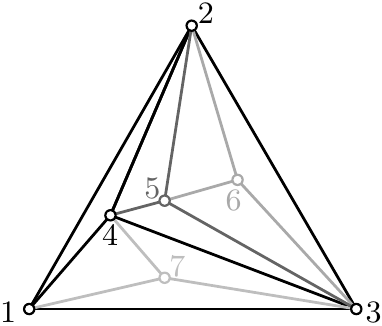}}
  \hfil
  \subcaptionbox{\label{fig:treedecomposition}}{\includegraphics[scale=.9,page=2]{tree_decomposition}}
  \hfil
  \subcaptionbox{\label{fig:stackingtree}}{\includegraphics[scale=.9, page=3]{tree_decomposition}}
  \hfil
  \subcaptionbox{\label{fig:grouping}}{\includegraphics[scale=.6]{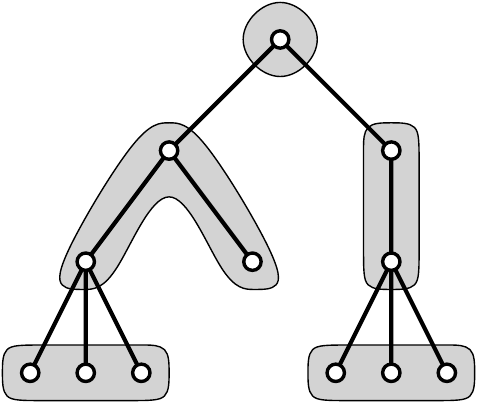}}
  \caption{(a) A planar 3-tree~$G$, (b) the tree $T_G$, and (c) the stacking tree of~$G$. A group partition
  of a different stacking tree (d).}
  \label{}
\end{figure}

The \emph{stacking tree} of $G$ is obtained by deleting all leaves in $T_G$;
see Fig.~\ref{fig:stackingtree}. We maintain the information
which vertex got stacked in which (temporary) face by an auxiliary structure.
By this, we can reconstruct $G$ from its stacking tree.
Next, we construct a partition (called \emph{group partition}) of the~$n-3$ nodes of the stacking tree.
We refer to the sets in a group partition as \emph{groups}. Each group (with one exception) will be of one of the following types.
\begin{enumerate}[label=(\Roman*)]
  \item A group of type I contains a node with one child.
  \item A group of type II contains a node with two children.
  \item A group of type III contains three siblings (but not their parent).
\end{enumerate}

We can construct such a partition by iteratively processing an arbitrary
deepest leaf~$v$ in the stacking tree (see Fig.~\ref{fig:grouping} for an example). If~$v$ has no sibling, then we group~$v$ and its parent to a type~I group; if $v$ has one sibling, then we group~$v$, its
sibling, and its parent to a type~II group; and if $v$ has two siblings,
we group~$v$ and its two siblings to a type~III group. Then, we remove all
vertices of the created group from the stacking tree and repeat. We stop when
either each vertex of the stacking tree is grouped, or only the root 
(corresponding to the first stacked vertex) remains. In the former case,
we create another group that is empty, and in the latter case, we 
create another group that contains the root only. Let~$g_1,\ldots,g_k$ be the groups of the partition in reverse order, such that $g_k$
was created first. The planar 3-tree
associated with the stacking tree obtained by 
the groups $g_1,\ldots,g_i$ , $1\le i\le k$, is named $G_i$.
If~$g_1$ is empty, we set~$G_1=K_3$.
We denote by $|G_i|$ the number of vertices in~$G_i$, and by~$|g_i|$ the number
of nodes in~$g_i$.

We will now create a path cover by iteratively adding the groups to the 
stacking tree. Hereby we make use of the following observation, which
is due to the fact that every leaf in the stacking tree is 
in correspondence with a degree~3 vertex in the planar 3-tree.

\begin{observation}\label{obs:leafdegree3}
  In any path cover of a planar 3-tree, for every leaf~$v$ in its corresponding
  stacking tree, there is a path with an endpoint in~$v$.
\end{observation}

Next, we show how to obtain a path cover for~$G$.

\begin{lemma}\label{lem:addgroup}
  Let~$G$ be a planar 3-tree, and let~$\alpha$, $\beta$, and~$\gamma$ be the number of type~I, type~II, and 
  type~III groups in some group partition.
  We can construct a path cover for~$G$ of size at most 
  $\alpha+2\beta+\gamma+2$ in linear time.
\end{lemma}
\begin{proof}
  If the first group is empty, we can easily find a path cover of 
  $G_1=K_3$ of size~2. Otherwise, the first
  group contains exactly the root of the stacking tree. It is an easy
exercise to find a path 
  cover of $G_1=K_4$ of size~2. We continue with
  adding the groups $g_2,\ldots,g_k$ in order one by one. 
  Let~$\pc_i$ be a path cover of~$G_i$, $1\le i<k$, of size~$p_i$.
  We now show how to get a path cover~$\pc_{i+1}$ of~$G_{i+1}$ of size
  $p_{i+1}$ such that~$p_{i+1}=p_i+1$ if~$g_{i+1}$ is of type~I or~III,
  and~$p_{i+1}=p_i+2$ if~$g_{i+1}$ is of type~II, which proves the bound in the 
  statement of the lemma.
  
  \setcounter{casecounter}{0}
  \ccase{c:typeI} The group~$g_{i+1}=\{u,v\}$ is of type~I. Assume that $v$ is a child of~$u$. 
  We take one path from~$\pc_i$ that contains an edge~$e$ that is incident to a face that contains~$u$ but not~$v$.
  We substitute~$e$ from this path by a subpath that visits 
  the two new vertices. Then, we add a path that covers~$e$ and the remaining
  three added edges as shown in Fig.~\ref{fig:TypeI}. This adds~1 new path, so we have
  $p_{i+1}=p_i+1$.
  
  \ccase{c:typeII} The group~$g_{i+1}=\{u,v,w\}$ is of type~II. We first mimic the procedure 
  of Case~\ref{c:typeI} and add 2 of the new vertices including the additional path~$P_I$. Let~$w$ be the 
  remaining vertex.
  Since every face created by stacking~$u$ and~$v$ is incident to an endpoint of~$P_I$, 
  we can extend~$P_I$ (after stacking~$w$) to~$w$. We add then a second
  new path for the other~2 edges incident to~$w$; see Fig.~\ref{fig:TypeII}. 
  This adds 2 new paths, so we have $p_{i+1}=p_i+2$.
  
  \ccase{c:typeIII} The group~$g_{i+1}=\{u,v,w\}$ is of type~III.
  The parent~$q$ of the three new vertices is a degree~3 node in~$G_i$, so some 
  path~$P$ of~$\pc_i$ ends in it. We remove the last edge~$e=(r,q)$ from~$P$.
  Two of the new vertices, say~$u$ and~$v$, share a face with~$e$.
  We extend~$P$ starting from~$r$ through~$u$, then through~$q$, and then
  to~$w$. Then, we take the path~$P'$
  from~$\pc_i$ that contains the edge~$e'=(r',q)$ that shares a face with~$w$ and~$v$. 
  We replace~$e'$ in~$P'$ by the edges
  $(r',v)$ and $(v,q)$. Finally, we cover~$e$, $e'$, and the remaining added
  edges by a single new path; see Fig.~\ref{fig:TypeIII}. This adds 1 new path,
  so we have $p_{i+1}=p_i+1$.
  
  Since the first group is covered by~2 paths, this proves the lemma.
  
  For the runtime, observe that   adding a group only takes constant time.
  Thus, it is bounded by creating the sequence of groups. To this end, we have
  to find the deepest leaf in every step. We can sort all nodes by their depth
  in linear time with e.g. Counting Sort, since the depth of the nodes is
  bounded by~$n$.   
  The stacking tree itself can be easily obtained in linear time from $T_G$ where computing  $T_G$ also takes  linear time \cite{Bodlaender96}.
  This concludes the proof.
    
  \begin{figure}[tb]
    \begin{subfigure}[c]{0.3\textwidth}
      \centering
      \includegraphics[]{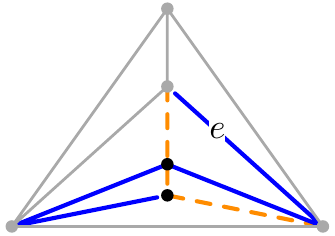}
      \subcaption{Type I}
      \label{fig:TypeI}
    \end{subfigure}
    \hfill
    \begin{subfigure}[c]{0.3\textwidth}
      \centering
      \includegraphics[]{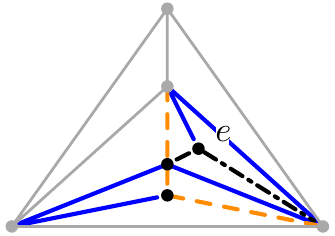}
      \subcaption{Type II}
      \label{fig:TypeII}
    \end{subfigure}  
    \hfill  
    \begin{subfigure}[c]{0.3\textwidth}
      \centering
      \includegraphics[]{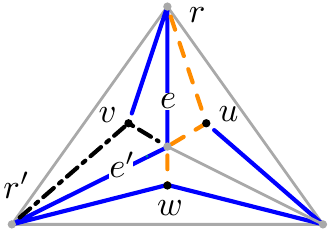}
      \subcaption{Type III}
      \label{fig:TypeIII}
    \end{subfigure}
    \caption{Illustration of the newly added group in Lemma~\ref{lem:addgroup}.}
    \label{fig:Types}
  \end{figure}  
\end{proof}
                                                                                 
A \emph{full planar 3-tree} is a planar 3-tree whose stacking tree is a
proper ternary tree, that is, there are no degree~2 nodes in the
stacking tree. In this case, 
the group partition uses only groups of type~III (except~$g_1$).
Hence, we have~$\alpha=\beta=0$ and~$\gamma=k-1$. Since the root cannot be
involved in a group of type~III,~$g_1$ will contain the root. Hence, $|g_1|=1$
and we have $n-3=\sum_{i=1}^k |g_i|=1+3(k-1)=3k-2$ and therefore $n=3k+1$. 
Thus, we can  create a path cover with at most $\gamma+2=k+1=(3k+3)/3 = \lceil n/3\rceil$
groups of type~III, leading us to the following proposition.

\begin{proposition}
  Any full planar 3-tree admits a path cover of size at 
  most~$\lceil n/3\rceil$. 
\end{proposition}

A planar 3-tree is called \emph{serpentine} if each stacking operation takes 
place on one of the three faces that were just created. The stacking tree of such a graph is a path.
Here, the group partition gives only groups of type~I
(except~$g_1$), which yields $\beta=\gamma=0$ and $\alpha=k-1$.
We have $n-3=\sum_{i=1}^k |g_i|=|g_1|+2(k-1)\ge 2k-2$. Thus, we have $n\ge 2k+1$ and we can 
create a path cover with at most $\alpha+2=k+1\le \lceil n/2\rceil$
groups of type~I, leading us to the following proposition.

\begin{proposition}
  Any serpentine planar 3-tree admits a path cover of size at most 
  $\lceil n/2\rceil$. 
\end{proposition}

The worst case for our algorithm is that the group partition only consists
of groups of type~II. In this case, we have $\alpha=\gamma=0$ and $\beta=k-1$.
We have $n-3=\sum_{i=1}^k |g_i|=|g_1|+3(k-1)\ge 3k-3$, which yields $n \ge 3k$. 
Thus, we can 
create a path cover with at most $2\beta+2=2k\le \lfloor 2n/3\rfloor$
groups.
Note that this bound was already proven by Dean and Kouider~\cite{DK00}
using Lov\'asz' construction.
However, we can combine our algorithm with the result of Dean and Kouider
to achieve a better bound.

\begin{theorem}\label{thm:planar3tree}
  Any planar 3-tree admits a path cover of size at 
  most~$\lfloor 5n/8\rfloor$.
\end{theorem}
\begin{pf}
  Let~$\alpha$, $\beta$, and~$\gamma$ be the number of type~I, type~II, and 
  type~III groups in the group partition of a stacking tree.
  By the size of the groups, we have that 
  \[
    n \ge 3+2\alpha+3\beta+3\gamma \ge 3+2\alpha+3\beta+2\gamma,
  \]
  which can be rephrased as $2(\alpha+\gamma)\le n - 3\beta - 3$.
  If~$\beta \le n/4-1$, then by Lemma~\ref{lem:addgroup} we can find a path cover with size at most 
  \begin{align*}
    \alpha+2\beta+\gamma+2 & \le 2\beta + \frac{n - 3\beta - 3}{2} + 2
    = \frac{n + \beta+ 1}{2} 
     \le \frac{n}{2} + \frac{n}{8}  
    \le 
    \frac{5n}{8}
    .
  \end{align*}

  Recall that by the bound of Dean and Kouider every graph has a path cover of size 
  $n_{\odd}/2+\lfloor 2n_{\even}/3\rfloor$, where
  $n_{\odd}$ is the number of odd-degree vertices and 
  $n_{\even}$ is the number of even-degree vertices~\cite{DK00}.
  If~$\beta > n/4 - 1$, then we apply their construction. 
  Note that the number of leaves in any tree exceeds 
  the number of its nodes with degree higher than~2. Since any group of type~II
  contains at least one vertex of degree at least~3 (the parent), we know that
  the number of leaves in the stacking tree is at least~$\beta+1$. By
  Observation~\ref{obs:leafdegree3}, each leaf in the stacking tree represents
  a degree-3 vertex, so we have~$n_{\odd} \ge \beta + 1  > n/4$.
  Hence, the construction of Dean and Koudier yields a path cover of size at most
  \begin{align*}
    \frac{n_{\odd}}{2}+\left \lfloor\frac{2n_{\even}}{3}\right\rfloor
     \le \frac{n_{\odd}}{2}+\frac{2n-2n_{\odd}}{3} 
    < \frac{16n}{24} - \frac{n}{24} 
    \le 
    \frac{5n}{8} 
    .\tag*{\qed} 
  \end{align*}
\end{pf}
\vspace*{-1.5ex}
\subsubsection*{Acknowledgements.} We thank Jens M. Schmidt for helpful discussions.


\clearpage

\bibliographystyle{abbrv}
\bibliography{gallai}

\end{document}